\documentclass[review]{elsarticle}

\usepackage{lineno,hyperref}
\modulolinenumbers[5]

\usepackage[utf8]{inputenc}
\usepackage{graphicx}
\usepackage{amsmath}
\usepackage{amsfonts}
\usepackage{amssymb}
\usepackage[all]{xy}
\usepackage{array}
\usepackage{enumerate}
\usepackage{url}
\usepackage{xcolor}
\usepackage{MnSymbol} 
\usepackage{hyperref}

\usepackage{tikz}
\usetikzlibrary{snakes}
\usetikzlibrary{arrows}
\usepackage{tikz-cd}
\usetikzlibrary{cd}
\usepackage{stackengine}
\setstackEOL{\\}
\tikzset{degil/.style={
		decoration={markings,
			mark= at position 0.5 with {
				\node[transform shape] (tempnode) {$\backslash$};
			}
		},
		postaction={decorate}
	}
} 

\newtheorem{theorem}{Theorem}[section]

\newtheorem{corollary}[theorem]{Corollary}

\newtheorem{example}[theorem]{Example}

\newtheorem{lemma}[theorem]{Lemma}

\newtheorem{proposition}[theorem]{Proposition}
\newtheorem{remark}[theorem]{Remark}

\newtheorem{theorem*}{Theorem}
\newenvironment{proof}{\medskip\emph{Proof.}}{\hfill$\Box$\medskip}

\newcommand{\DD}{\mathcal{D}}

\newcommand{\LL}{\mathcal{L}}
\newcommand{\RR}{\mathcal{R}}

\def\id{\mathrm{id}}

\journal{Journal of \LaTeX\ Templates}

\begin{document}

\begin{frontmatter}

\title{Skew lattices and set-theoretic solutions of the Yang-Baxter equation}

\author[address]{Karin Cvetko-Vah\corref{mycorrespondingauthor}}
\cortext[mycorrespondingauthor]{Corresponding author}
\ead{karin.cvetko\symbol{64}fmf.uni-lj.si} \author[address2]{Charlotte Verwimp}
\ead{Charlotte.Verwimp@vub.be}
\address[address]{University of Ljubljana,
	Faculty of Mathematics and Physics, \newline
	Jadranska 19,
	SI-1001, Ljubljana,
	SLOVENIA.}
\address[address2]{Vrije Universiteit Brussel, Department of
	Mathematics \newline 
	Pleinlaan 2, 1050 Brussel, BELGIUM.}


%
%

\begin{abstract}
In this paper we discuss and characterize several set-theoretic solutions of the Yang-Baxter equation obtained using skew lattices, an algebraic structure that has not yet been related to the Yang-Baxter equation. Such solutions are degenerate in general, and thus different from solutions obtained from braces and other algebraic structures.  

Our main result concerns a description of a set-theoretic solution of the Yang-Baxter equation, obtained from an arbitrary skew lattice. We also provide a construction of a cancellative and distributive skew lattice on a given family of pairwise disjoint sets.
\end{abstract}


\end{frontmatter}


\section*{Introduction}

The study of the Yang-Baxter equation is a fundamental topic in theoretical physics and fundamental mathematics. It laid foundation to the theory of quantum groups, and is used in many other fields like Hopf algebras and statistical mechanics. As finding arbitrary solutions is relatively complicated, Drinfeld \cite{D:1992} proposed to focus on a smaller set of solutions, the so-called set-theoretic solutions of the Yang-Baxter equation. 
A \emph{set-theoretic solution} of the Yang-Baxter equation is a pair $(X,r)$ where $X$ is a non-empty set and $r: X \times X \rightarrow X \times X: (x,y) \mapsto (\lambda_x(y),\rho_y(x))$ is a map satisfying
\begin{equation}
(r\times \text{id})\circ (\text{id} \times r)\circ (r\times \text{id})=(\text{id} \times r)\circ (r\times \text{id})\circ (\text{id} \times r).
\end{equation}
We briefly call this a \emph{solution}. An easy example of a solution is the twist map, defined as $r(x,y)=(y,x)$, for all $x,y \in X$. Moreover, this solution is \emph{involutive} (i.e. $r^2= \text{id}_{X^2}$) and \emph{non-degenerate}, which means that the solution is both \emph{left non-degenerate} (i.e. $\lambda_x: X \rightarrow X$ is bijective, for all $x \in X$) and \emph{right non-degenerate} (i.e. $\rho_x:X \rightarrow X$ is bijective, for all $x \in X$). If a solution is neither left nor right non-degenerate, then it is called \emph{degenerate}. 
Next to the important class of involutive solutions, there is another relevant class of set-theoretic solutions where the map $r$ is idempotent (i.e. $r^2=r$). These solutions are called \emph{idempotent}. A set-theoretic solution $(X,r)$ is said to be \emph{cubic} if $r^3=r$. Note that the class of cubic solutions contains the class of idempotent solutions and the class of involutive solutions.


Non-degenerate (involutive) set-theoretic solutions of the Yang-Baxter equation have been studied intensively over the past years, see for example \cite{ESS:1999,GIVDB:1998,LYZ:2000,S:2000}. The study of set-theoretic solutions of the Yang-Baxter equation became even more attractive since the discovery of related algebraic structures. In \cite{R:2005,R:2007}, Rump introduced cycle sets and braces to investigate non-degenerate involutive solutions. 
Later braces were generalized by Guarnieri and Vendramin \cite{GV:2017} to study non-degenerate solutions that are not necessarily involutive. Catino, Colazzo, and Stefanelli \cite{CCS:2017}, and Jespers and Van Antwerpen \cite{JVA:2019} introduced (left cancellative) left semi-braces to deal with solutions that are not necessarily non-degenerate or solutions that are idempotent or cubic. Some of these solutions are degenerate. However, in general, not much is known about degenerate solutions and how to construct them.

Next to involutive solutions, another important class of set-theoretic solutions of the Yang-Baxter are the idempotent solutions. Using idempotent solutions and graphical calculus from knot theory, Lebed in \cite{L:2017} provides a unifying tool 
to deal with several diverse algebraic structures, such as free and free commutative monoids, factorizable monoids, plactic monoids, Young tableaux and distributive lattices. The latter are important in the following way.
A \emph{lattice} is a triple $(L,\land,\lor)$ where both $(L,\land)$ and $(L,\lor)$ are commutative bands (i.e. both operations are commutative, associative and idempotent) satisfying the absorption laws
$x\land(x\lor y) =x=x \lor (x\land y),$ 
for all $x,y\in L$. To each lattice $(L,\land,\lor)$ one may associate an idempotent map $r:L\times L \rightarrow L\times L$ defined as $r(x,y)=(x \land y, x \lor y)$. Moreover, $(L,r)$ is a solution of the Yang-Baxter equation if and only if $L$ is a distributive lattice, i.e. for any $x,y,z \in L$, $x\land(y\lor z) =(x\land y)\lor(x\land z)$ and $x\lor(y\land z) =(x\lor y)\land(x\lor z)$.
\newline
\newline
In this paper we generalize the idea of using lattices to obtain idempotent and cubic set-theoretic solutions of the Yang-Baxter equation. In particular, we use a recently intensively studied algebraic structure (\cite{Cvet07, Cv2011, canc, update, dist, L1, L3, L4, L6, S:2000}), called a skew lattice, to produce set-theoretic solutions of the Yang-Baxter equation. Furthermore, these solutions are degenerate in general.
\newline
\newline
The paper is organized as follows. The first section contains some preliminaries on skew lattices. We include basic knowledge of skew lattices which is needed for the rest of the paper.

As we try to find solutions using skew lattices, it is important to have many examples of skew lattices or even be able to construct skew lattices on given sets. In Section 2 we provide such a construction, where we define a skew lattice structure on a given family of pairwise disjoint sets.

Section 3 contains the main result of this paper, namely the description of an idempotent set-theoretic solutions of the Yang-Baxter equation associated to an arbitrary skew lattice.
\begin{theorem}
	Let $(S,\land,\lor)$ be a skew lattice. Then the map defined by $r(x,y)=((x \land y) \lor x, y)$ is an idempotent set-theoretic solution of the Yang-Baxter equation.
\end{theorem}

Inspired by the solution obtained from a distributive lattice, in Section 4 we define and describe \emph{strong distributive solutions} of the Yang-Baxter equation. These are skew lattices $(S, \land, \lor)$ such that the map $r:S \times S \rightarrow S \times S: (x,y) \mapsto (x \land y, x \lor y)$ is a set-theoretic solution of the Yang-Baxter equation. We prove that skew lattices, which are simultaneously \emph{strongly distributive} (i.e. they satisfy $x\land (y\lor z)=(x\land y)\lor (x\land z)$, $(x\lor y)\land z=(x\land z)\lor (y\land z)$) and \emph{co-strongly distributive} (i.e. they satisfy $x\lor (y\land z)=(x\lor y)\land (x\lor z)$, $(x\land y)\lor z=(x\lor z)\land (y\lor z)$), are always strong distributive solutions. However, the converse is not true.

The final section is devoted to some other distributive solutions. In particular we define left, right and weak distributive solutions and describe them in terms of properties of skew lattices introduced in the first section. A surprising result is that for symmetric skew lattices, being a left, right or weak distributive solution are all equivalent notions. For lattices, it is even equivalent to being a strong distributive solution.

\section{preliminaries}

\subsection{Skew lattices}
Skew lattices were first introduced by Pascual Jordan in \cite{jordan}. Later the definition of a skew lattice was slightly adopted by Jonathan Leech, who initiated the modern study of these structures in \cite{L1}. 
A \emph{skew lattice} is a set $S$ endowed with a pair of idempotent and associative operations  $\land$ and $\lor$ which satisfy the absorption laws
\[
x\land (x\lor y)=x=x\lor (x\land y)\text{ and } (x\land y)\lor y=y=(x\lor y)\land y.
\]  
By a result of Leech \cite{L1} the following pair of dualities hold in any skew lattice,
\begin{align*}
x\land y = x \text{ iff } x\lor y = y, \\
x\land y = y \text{ iff } x\lor y = x.
\end{align*}

Recall \cite{How} that a \emph{band} is a semigroup of idempotents. A band is called \emph{regular} if it satisfies the identity $axaya=axya$. A complete list of varieties of bands can be found in \cite{petrich}. Leech \cite[Theorem 1.15]{L1} proved that given a skew lattice $S$, both semigroups $(S, \land)$ and $(S, \lor)$ are regular bands, i.e. the identities 
\begin{align}
a\land x\land a\land y\land a &= a\land x\land y\land a,\label{eq:reg1}
\\ a\lor x\lor a\lor y\lor a &= a\lor x\lor y\lor a,\label{eq:reg2}
\end{align}
are always satisfied.


\emph{Green's equivalence relations} $\LL, \RR, \DD, \mathcal H$ and $\mathcal J$ are fundamental tools in the theory of semigroups. Only the first three are relevant in the study of bands, as for bands $\DD=\mathcal J$ and $\mathcal H$ is the diagonal relation. We refer the reader to \cite{How} for the definition of Green's relation for general semigroups. By \cite[Lemma I.7.1]{petrich-cr} in the case of bands the definitions simplify as follows:
\[
\begin{array}{rcl}
x \LL y & \text{ iff } & xy =x, \, yx=y,\\
x \RR y & \text{ iff } & xy =y, \,yx=x,\\
x \DD y & \text{ iff } & xyx =x, \, yxy=y.
\end{array}
\]
On any semigroup $S$, relation $\LL$ is a right congruence ($a\,\LL\, b$ implies $ac\,\LL\,bc$ for any $c\in S$), relation $\RR$ is a left congruence ($a\,\RR\, b$ implies $ca\,\RR\, cb$ for any $c\in S$), see \cite[Proposition 2.1.2]{How}. Moreover, each $\DD$-class is a union of $\LL$-classes and also a union of $\RR$-classes.  The intersection of an $\LL$-class with an $\RR$ class is either empty or it is an $\mathcal H$-class. Because of this property a $\DD$-class is sometimes visualized as an 'eggbox', with rows corresponding to $\RR$-classes, columns corresponding to $\LL$-classes, and intersections of rows and columns corresponding to $\mathcal H$-classes (with the latter being singletons in the case of bands). A semigroup satisfying the identity $x\land y=x$ is called a \emph{left-zero semigroup}, and a semigroup satisfying $x\land y=y$ is called a \emph{right-zero semigroup}.

A skew lattice $(S,\land,\lor)$ is given by a pair of bands $(S,\land)$ and $(S,\lor)$. We denote the corresponding Green's relations by $\LL_\land$, $\RR_\land$, $\DD_\land$, and $\LL_\lor$, $\RR_\lor$, $\DD_\lor$, respectively.
By Leech's First Decomposition Theorem \cite[Theorem 1.7]{L1},  Green's relations $\DD_\land$ and $\DD_\lor$ on any skew lattice $S$ coincide (and are thus denoted simply by $\DD$), this relation $\DD$ is a congruence, $S/\DD$ is the maximal lattice image of $S$, and each $\DD$-class is a \emph{rectangular skew lattice}, meaning that it is a skew lattice  satisfying the additional identities $x\land y\land z=x\land z$ and $x\lor y=y\land x$. In general,  Green's relation $\DD$ on a semigroup need not be a congruence. However, by Clifford-McLean Theorem, relation $\DD$ is a congruence on any band, and any band factorizes as a \emph{semilattice}
(commutative band) of \emph{rectangular bands} (characterized by the identity $xyz=xz$), see \cite[Theorem3]{clifford} and \cite[Theorem 1]{mclean}. In what follows we denote by $\DD_x=\{t \in S \mid x \,\DD\, t\}$ the $\DD$-class of an element $x$ of $S$.

Moreover, on a skew lattice $S$ we obtain $\RR_\land=\LL_\lor $ (which is denoted by $\RR$), $\RR_\lor=\LL_\land$ (which is denoted by $\LL$). 
A skew lattice is called \emph{left handed} if $\LL=\DD$, and it is called \emph{right handed} if $\RR=\DD$. A skew lattice is thus left handed if and only if it satisfies the identity $x\land y\land x=x\land y$, or equivalently, $x\lor y\lor x=y\lor x$, and it is right handed if and only if it  satisfies  $x\land y\land x = y\land x$, or equivalently, $x\lor y\lor x=x\lor y$. By Leech's Second Decomposition Theorem for skew lattices  \cite[Theorem 1.15]{L1}, Green's relations $\RR$ and $\LL$ are congruences on any  skew lattice $S$, and $S$ factors as a fiber product of a left handed skew lattice $S/\RR$ (called the \emph{left factor} of $S$) by a right handed skew lattice $S/\LL$ (called the \emph{right factor} of $S$) over their common maximal lattice image. Leech's Second Decomposition Theorem  is a skew lattice version of the Kimura Theorem for regular bands \cite[Theorem 4]{kim}. 
In particual, if $S$ is a rectangular skew lattice (which is equivalent to $S$ having exactly one $\DD$-class), then $S$  factors as direct product $S\cong L\times R$ of a {left-zero semigroup} $L$ (satisfying $x\land y=x$) by a {right-zero semigroup} $R$ (satisfying $x\land y=y$).

The \emph{natural preorder} on a skew lattice $S$ is given by $x\preceq y$ if and only if $x\land y\land x=x$, or equivalently, $y\lor x\lor y=y$. Note that $x\preceq y$ together with $y\preceq x$ is equivalent to $x\,\DD\, y$. Moreover, $x\preceq y$ holds in $S$ if and only if $\DD_x\leq \DD_y$ holds in the lattice $S/\DD$. The \emph{natural partial order} is given by $x\leq y$ if and only if $x\land y=x=y\land x$, or equivalently, $x\lor y=y=y\lor x$. Note that $x\leq y$ implies $x\preceq y$.

The following result of \cite[Corollary 3]{Cvet07} is used frequently through this paper. 
\begin{theorem}\label{thm:lrhanded}
	A skew lattice satisfies an identity or an equational implication if and only if both its left and right factor satisfy this identity or equational implication. 
\end{theorem}
Thus, to prove a result for a skew lattice $S$ it is enough to prove it first assuming that $S$ is left handed (i.e. for any $x,y \in S$, $x\land y\land x=x\land y$, or equivalently, $x\lor y\lor x=y\lor x$) and next assuming that $S$ is right handed (i.e. for any $x,y \in S$, $x\land y\land x=y\land x$, or equivalently, $x\lor y\lor x=x\lor y$).

\subsection{Varieties of skew lattices}
A class of algebras is called a \emph{variety} if it is closed under homomorphic images, substructures and direct products. By Birkhoff's Theorem \cite[Theorem 11.9]{BS} a class of algebras is a variety if and only if it is equationally defined, i.e. if it is defined by a set of identities. 

Amongst all properties of skew lattices, symmetry is a very fundamental one. A skew lattice is called \emph{symmetric}\label{symmetric} if for any $x,y\in S$, $x\land y=y\land x$ if and only if $x\lor y=y\lor x.$ Symmetric skew lattices form a variety by a result in \cite[Subsection 2.3]{L1}.



More varieties are defined throughout the paper. Below we provide a diagram to give an overview on the varieties that are used in this paper. 

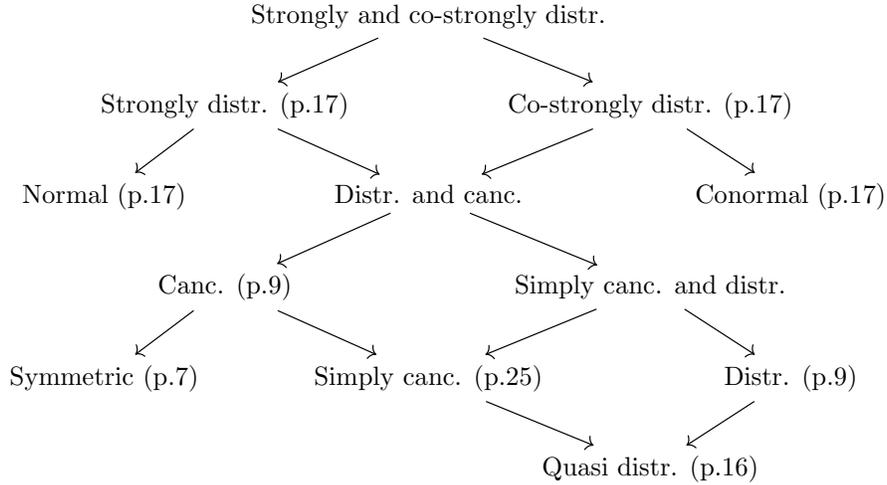
\[\begin{tikzcd}[column sep = -1.6cm, row sep = 0.6cm]
& & \text{Strongly and co-strongly distr.} 
\arrow{dl} \arrow{dr} & &
\\ 
& \text{Strongly distr. (p.\pageref{strongly distributive})} \arrow{dr} \arrow{dl} & & \text{Co-strongly distr. (p.\pageref{co-strongly distributive})} 
\arrow{dl} \arrow{dr} & 
\\ 
\text{Normal (p.\pageref{normal})} & & \text{Distr. and canc.} \arrow{dl} \arrow{dr} & & \text{Conormal (p.\pageref{conormal})}
\\
& \text{Canc. (p.\pageref{cancellative})} \arrow{dl} \arrow{dr} & & \text{Simply canc. and distr.} \arrow{dl} \arrow{dr} & 
\\ 
\text{Symmetric (p.\pageref{symmetric})} & & \text{Simply canc. (p.\pageref{simply cancellative})} \arrow{dr} & & \text{Distr. (p.\pageref{distributive})} \arrow{dl} 
\\
& & & \text{Quasi distr. (p.\pageref{quasi-distributive})} &
\end{tikzcd}\]
The reader might find the diagram above useful when thinking about the partial order between different varieties of skew lattices. Note that there might be other varieties of skew lattices that are not shown in the diagram and lie in between the varieties that are shown. The diagram therefore provides the information regarding the partial order on the set of listed varieties, but it does not imply for instance that quasi-distributive skew lattices form the join of the variety of  simply cancellative and the variety of distributive skew lattices.

\subsection{Skew lattices in rings}
Let $R$ be a ring
and  $E(R)=\{e\in R:e^2=e \}$ the set of all idempotents in $R$. We say that a subset
$S\subset E(R)$ is a \emph{multiplicative band} if $xy\in E(R)$ for all $x,y\in E(R).$
Given a multiplicative band $S$ in a ring $R$ there are two natural ways to define the join operation on $E(R)$ (letting the meet be  multiplication):
\begin{itemize}
	\item[(i)] \emph{the quadratic join:} $x\circ y=x+y-xy$;
	\item[(ii)] \emph{the cubic join:} $x\nabla y=\left(  x\circ
	y\right)  ^{2}=x+y+yx-xyx-yxy.$
\end{itemize}
In general, given $x,y\in E(R)$, $x\circ y$
need not
be an idempotent. On the other hand, whenever $E(R)$ is multiplicative we obtain $x\nabla y\in E(R)$ for all $x,y\in E(R)$. However, $\nabla$ need not be associative. Note that if $x\circ y\in E(R)$  then $x\nabla y=x\circ y$.

Leech \cite{L1} proved the following pair of results,
\begin{itemize}
	\item[(i)] If $(S,\cdot)$ is a multiplicative band that is
	closed under $\circ$, then $S$ is a skew lattice, called \emph{a quadratic
		skew lattice.}
	\item[(ii)]
	If $(S,\cdot)$ is a multiplicative band  that is also closed under $\nabla$, with  $\nabla$ being associative on $S$, then $S$ is
	a skew lattice, called \emph{a cubic skew lattice.}
\end{itemize}

Moreover, the following was proven in \cite{L1} and \cite{L4}.

\begin{itemize}
	\item[(i)] Every maximal right [left] regular band in a ring forms a quadratic skew lattice.  Every  right [left] regular band in a ring generates a quadratic skew lattice. (A band is called \emph{right regular} if it satisfies the identity $xyx=yx$; it is called \emph{left regular} if it satisfies the identity $xyx=xy$.)
	\item[(ii)]
	Every maximal normal band in a ring forms a  normal cubic skew lattice. Every normal band in a ring generates a normal  cubic skew lattice.
\end{itemize}

Quadratic as well cubic skew lattices in rings are distributive and cancellative by  \cite{L1}.

\section{Construction}

In this section we show a construction of a class of skew lattices that are not lattices (i.e. a skew lattice where both operations $\land$ and $\lor$ are commutative). 
In particular, we construct a skew lattice on an arbitrary family of pairwise disjoint sets. The constructed skew lattice is distributive and cancellative, where
a skew lattice is called \emph{(fully) cancellative}\label{cancellative} if it satisfies the following pair of implications 
\begin{align}
x \lor y = x \lor z, x \land y = x \land z \Longrightarrow y = z,  \label{C1}\\
x \lor z = y \lor z, x \land z = y \land z \Longrightarrow x = y. \label{C2}
\end{align}
Cancellative skew lattices form a variety by a result in \cite{canc}. Following Leech \cite{L1}, a skew lattice is called \emph{distributive}\label{distributive} if it satisfies the following pair of identities
\begin{align}
x\land (y\lor z)\land x &= (x\land y\land x)\lor (x\land z\land x),\label{eq:D1}
\\ x\lor (y\land z)\lor x &= (x\lor y\lor x)\land (x\lor z\lor x).\label{eq:D2}
\end{align}
Unlike the situation for lattices, the identities \eqref{eq:D1} and \eqref{eq:D2} are in general independant. However, by Spinks' Theorem \cite{spinks}, \eqref{eq:D1} and \eqref{eq:D2} are equivalent for symmetric skew lattices. 
Note that for lattices distributivity is defined differently. In fact, as we will see later, there are several ways to generalize the notion of distributivity to the non-commutative setting. 
\newline
\newline
Let $(I,\leq)$ be a totally ordered set and $S= \bigcupdot_{i \in I} A_i$ the disjoint union of a family of pairwise disjoint sets. We define operations $\land$ and $\lor $ on $S$ as follows. Take any $x,y\in S$ and let $i,j\in I$ be such that $x\in A_i$, $y\in A_j$. Define
\[
x\land y=
\left\{
\begin{array}{ll}
x & \text{if } i<j \\
y & \text{if } j\leq i
\end{array},
\right.
\qquad
x\lor y=
\left\{
\begin{array}{ll}
y & \text{if } i<j \\
x& \text{if } j\leq i
\end{array}.
\right.
\]

\begin{proposition}\label{prop:construction}
	Let $(I,\leq)$ be a totally ordered set and $S= \bigcupdot_{i \in I} A_i$ a disjoint union of a family of pairwise disjoint sets. Then, $(S;\land, \lor)$ is a distributive and  cancellative skew lattice, where $\land$ and $\lor$ are defined as above.
\end{proposition}


\begin{proof}
	Note that $x\land (y\land z)$ and $(x\land y)\land z$ both reduce to the minimal element of $\{x,y,z\}$ that appears most right in the expression $x\land y\land z$. Similarly,  $x\lor (y\lor z)$ and $(x\lor y)\lor z$ both reduce to the maximal element of $\{x,y,z\}$ that appears most left in the expression $x\lor y\lor z$. Hence, both $\land$ and $\lor$ are associative operations.
	
	To prove the absorption laws, let $x\in A_i$, $y\in A_j$. Assume first that $i<j$. Then $x\land (x\lor y)=x\land y=x$. On the other hand, if $i\geq j$ then $x\land (x\lor y)=x\land x=x$. The rest of the absorption laws are proven in a similar fashion.
	
	Note that given  $x\in A_i$ and $y\in A_j$, $x\,\DD\, y$ is equivalent to $i=j$. Moreover, $x\in A_i$ commutes with $y\in A_j$ for either of the operations $\land $, $\lor$, if and only if $i\neq j$.
	By construction, $S$ is a \emph{skew chain} which means that its maximal lattice image $S/\DD$ is totally ordered (in our case isomorphic to $I$). Skew chains are always cancellative by  \cite[Proposition 5]{Cvet06}. 
	
	It remains to prove that $S$ is distributive. Take any $x,y,z\in S$ with $i,j,k\in I$ such that $x\in A_i$, $y\in A_j$ and $z\in A_k$, and consider the elements $\alpha=x\land (y\lor z)\land x$, $\beta=(x\land y\land x)\lor (x\land z\land x).$ We consider the following cases:
	\begin{itemize}
		\item Assume $j=k$. Then  $y\,\DD\, z$ and thus $(x\land y\land x)\,\DD\, (x\land z\land x)$, which means that there exists $l\in I$ such that $x\land y\land x, x\land z\land x$ both lie in $A_l$. It follows that $\alpha=x\land y\land x$ and likewise $\beta=x\land y\land x$. 
		\item Assume $j \neq k$. Then $y$ and $z$ commute for either of the operations $\land$, $\lor$, with either $y\land z=z\land y=y$ and $y\lor z=z\lor y=z$, or $y\land z=z\land y=z$ and $y\lor z=z\lor y =y$. Thus by the definition of the natural partial order, we have either $y<z$ or $z<y$. If $y<z$ then also $x\land y\land x<x\land z\land x$. (By regularity (\ref{eq:reg1}), $x\land y\land x\land x\land z\land x=x\land y\land z\land x=x\land y\land x$.) It follows that $\alpha=x\land z\land x=\beta$. Similarly, if $z<y$ then $\alpha=x\land y\land x=\beta$.
	\end{itemize}
\end{proof}


\section{Solutions obtained from general skew lattices}
In this section we use arbitrary skew lattices to produce new set-theoretic solutions of the Yang-Baxter equation. These solutions are of idempotent type, and thus of importance as is shown in \cite{L:2017}.

Given elements $x,y$ in a skew lattice $S$ we define the \emph{lower update} of $x$ by $y$ as
\[
x\lfloor y \rfloor=(y\wedge x\wedge y)\vee x\vee (y\wedge x\wedge y).
\]
Denote $A=\DD_x$, $B=\DD_y$ and $M=A\land B$. The following properties of $x\lfloor y \rfloor$ were proved in \cite{update}:
\begin{itemize}
	\item $x\lfloor y \rfloor\in A$,
	\item $x\lfloor y \rfloor$ is the unique element of the coset $M\lor x\lor M$ such that $y\land x\land y\leq x\lfloor y \rfloor$,
	\item $x\lfloor y \rfloor \land y=y\land x\land y=y\land x\lfloor y \rfloor$.
\end{itemize}

Moreover, if $S$ is left handed then $x\lfloor y \rfloor= x\vee (y\wedge x)$; and if $S$ is right handed then $x\lfloor y \rfloor=(x\wedge y)\vee x$.
In order to be able to prove  Lemma \ref{lemma:update} below, we need to recall some further facts and definitions from skew lattice theory in the following remark.

\begin{remark}\label{remark1}\normalfont
	The geometric structure of skew lattices was studied in \cite{L6}. Given comparable $\DD$-classes $A, B$ in a skew lattice $S$ such that $A>B$ holds in the lattice $S/\DD$, a \emph{coset of $A$ in $B$} is a subset $A\land b\land A=\{a\land b\land a'\,|\, a,a'\in A\}\subseteq B$, where $b\in B$. Likewise, a \emph{coset of $B$ in $A$} is a subset $B\lor a\lor B=\{b\lor a\lor b'\,|\, b,b'\in B\}\subseteq A$, where $a\in A$. 
	
	Moreover, given any coset $B_j$ of $A$ in $B$ and any coset $A_i$ of $B$ in $A$, there exists a bijection $\varphi_{ji}:A_i\to B_j$ which maps an element $x\in A_i$ to the unique element $y\in B_j$ with the property $y\leq x$ w.r.t. natural partial order. In order to prove that a pair of elements of $A$ are equal it thus suffices to show that they lie in the same coset of $B$ in $A$ \emph{and} are both above the same element of $B$ w.r.t. natural partial order. Likewise,  a pair of elements of $B$ are equal if and only if they lie in the same coset of $A$ in $B$ and they are both below the same element of $A$ w.r.t. natural partial order.
\end{remark} 
We will apply the technique from Remark \ref{remark1} in the proof of Lemma \ref{lemma:update} below.

\begin{lemma}\label{lemma:update}
	Let $S$ be a skew lattice and $x,y,z\in S$. Then,
	\begin{equation}\label{eq:update}
	(x\lfloor y \rfloor)\lfloor y\lfloor z \rfloor \rfloor=x\lfloor y \lfloor z\rfloor \rfloor.
	\end{equation}
\end{lemma}

\begin{proof}
	Denote $A=\DD_x$, $B=\DD_y$ and $M=A\land B$.   Consider first the element $(x\lfloor y \rfloor)\lfloor y\lfloor z \rfloor \rfloor$. Being an (multiple) update of $x$, it must lie in $A$. In fact, it is the unique element of the coset $M\lor x\lor M$ in $A$ that is above $y\lfloor z \rfloor \land x\lfloor y \rfloor \land y\lfloor z \rfloor $ w.r.t. natural partial order. Likewise, $x\lfloor y \lfloor z\rfloor \rfloor$ also lies in $A$, and it is the unique element of $M\lor x\lor M$ that is above $y\lfloor z\rfloor \land x\land y\lfloor z\rfloor$. In order to prove \eqref{eq:update} it suffices to show that $u=v$, where $u=y\lfloor z \rfloor \land x\lfloor y \rfloor \land y\lfloor z \rfloor$ and $v=y\lfloor z\rfloor \land x\land y\lfloor z\rfloor $. 
	
	Both $u$ and $v$ are elements of $M$. We claim that they lie in a common coset of $B$ in $M$, i.e. that  $B\land u \land B = B \land v \land B$ holds.
	The coset $B \land u \land B$ contains the element $y \land u \land y = y \land y\lfloor z\rfloor \land x\lfloor y\rfloor \land y\lfloor z\rfloor \land y$.
	Using regularity (\ref{eq:reg1}), this equals to $y \land y\lfloor z\rfloor \land y \land x\lfloor y\rfloor \land y \land y\lfloor z\rfloor \land y$.
	Using $x\lfloor y\rfloor \land y = y \land x\lfloor y\rfloor = y \land x \land y$, the above equals
	$y \land y\lfloor z\rfloor \land y \land x \land y \land y\lfloor z\rfloor \land y$, which by regularity (\ref{eq:reg1}) simplifies to $y \land y\lfloor z\rfloor \land x \land y\lfloor z\rfloor \land y = y \land v \land y$, which is  an element of the coset $B \land v \land B$. 
	It follows the cosets $B \land u \land B$ and $B \land v \land B$ intersect, and are thus equal. 
	
	Finally, observe that $u$ and $v$ both lie below $ y\lfloor z\rfloor$, i.e. $ y\lfloor z\rfloor\land u=u=u\land  y\lfloor z\rfloor$ and $ y\lfloor z\rfloor\land v=v=v\land  y\lfloor z\rfloor$. It follows that $u=v$.
\end{proof}

\begin{theorem}\label{th:update}
	Let $(S,\land,\lor)$ be a  skew lattice. Then the map defined by $r(x,y)=((x \land y) \lor x, y)$ is an idempotent set-theoretic solution of the Yang-Baxter equation.
\end{theorem}

\begin{proof}
	Let $x,y,z \in S$. Then
	\begin{align*}
	&(r\times \id) (\id\times r)(r\times \id) (x, y,z) \\&=(r\times \id) (\id\times r)((x\land y) \lor x, y, z) 
	\\ &= (r\times \id)((x\land y) \lor x, (y\land z) \lor y, z)
	\\ &= ((((x\land y) \lor x) \land ((y\land z) \lor y)) \lor ((x\land y) \lor x), (y\land z) \lor y, z),
	\end{align*}
	and 
	\begin{align*}
	&(\id\times r)(r\times \id)(\id\times r) (x, y,z) \\&=(\id\times r)(r\times \id)(x, (y\land z) \lor y, z)
	\\ &= (\id\times r)((x \land ((y\land z) \lor y)) \lor x,(y\land z) \lor y, z)
	\\ &= ((x \land ((y\land z) \lor y)) \lor x,(((y\land z) \lor y) \land z) \lor ((y\land z) \lor y), z).
	\end{align*}
	
	We need to prove
	\begin{equation}\label{eq:sol-lower-update1}
	(((x\land y) \lor x) \land ((y\land z) \lor y)) \lor ((x\land y) \lor x)=(x \land ((y\land z) \lor y)) \lor x,
	\end{equation}
	and
	\begin{equation}\label{eq:sol-lower-update2}
	(y\land z) \lor y=(((y\land z) \lor y) \land z) \lor ((y\land z) \lor y).
	\end{equation}
	
	First assume that $S$ is left handed. Then, using absorption, \eqref{eq:sol-lower-update1} simplifies to $(x\land y)\lor x=(x\land y)\lor x$ (and even  to $x=x$). Similarly, \eqref{eq:sol-lower-update2} simplifies to $y=y$.
	
	Assume next that $S$ is right handed. Denote $A=\DD_x$, $B=\DD_y$ and $M=A\land B$.  Notice that \eqref{eq:sol-lower-update1} simplifies to \eqref{eq:update}, which holds by Lemma \ref{lemma:update}.
	
	If $S$ is right handed then  \eqref{eq:sol-lower-update2} simplifies to $y\lfloor z\rfloor = (y\lfloor z\rfloor\land z)\lor ((y\land z)\lor y)$. Notice that $ (y\lfloor z\rfloor\land z)\lor ((y\land z)\lor y)$ simplifies to $(z\land y\land z) \lor ((y\land z)\lor y)$. Using right handedness  the latter further simplifies to $(y\land z) \lor ((y\land z)\lor y)=(y\land z)\lor y=y\lfloor z\rfloor.$
	
	We proved that all left handed as well as all right handed skew lattices satisfy  identities \eqref{eq:sol-lower-update1} and \eqref{eq:sol-lower-update2}. By Theorem \ref{thm:lrhanded}, all skew lattices satisfy these identities.
	
	It remains to prove that the solution is idempotent. Again, by Theorem \ref{thm:lrhanded}, it is enough to prove this separately for left and right handed skew lattices. 
	
	Let $S$ be left handed. Then $(x\land y)\lor x=(x\land y\land x)\lor x$, which simplifies to $x$ by absorption  Thus: $r(x,y)=(x,y)=r^2(x,y).$
	
	If $S$ is right handed, then $r(x,y)=(x\lfloor y\rfloor ,y)$ and $r^2(x,y)=((x\lfloor y\rfloor)\lfloor y\rfloor, y)$. However, $(x\lfloor y\rfloor)\lfloor y\rfloor$ equals $x\lfloor y\rfloor$ as they both represent the unique element in the coset $M\lor x\lor M$ that is above $y\land x\land y$ w.r.t natural partial order.
\end{proof}

\begin{corollary}\label{cor:update}
	Let $(S, \land, \lor)$ be a skew lattice. The map $r(x,y)=(x\lfloor y\rfloor, y)$ is an idempotent set-theoretic solution of the Yang-Baxter equation.
\end{corollary}

\begin{proof}
	We obtain 
	\[(r\times \id) (\id\times r)(r\times \id) (x, y,z)=((x\lfloor y \rfloor)\lfloor y\lfloor z\rfloor\rfloor, y\lfloor z\rfloor ,z),
	\]
	and
	\[ (\id\times r)(r\times \id) (\id\times r)(x, y,z)=(x\lfloor y \lfloor z\rfloor\rfloor, (y\lfloor z\rfloor)\lfloor z\rfloor ,z).
	\]
	The equality of first components follows by Lemma \ref{lemma:update}. The equality of the second components holds because $y\lfloor z\rfloor$ and $(y\lfloor z\rfloor) \lfloor z\rfloor$ are both elements of the coset $M\lor y\lor M$ (where $M=\DD_{y\land z}$), that lie above $z\land y\land z$. Since such elements are unique, it follows that $y\lfloor z\rfloor=(y\lfloor z\rfloor) \lfloor z\rfloor$.
\end{proof}

Dually to the notion of a lower update, the notion of an \emph{upper update operation} was defined in \cite{update}. Using the upper update we obtain analogue results to Lemma \ref{lemma:update}, Theorem \ref{th:update} and Corollary \ref{cor:update}, yielding additional solutions.

\section{Strong distributive solutions of the Yang-Baxter equation}

In the previous section, we provided an idempotent set-theoretic solution of the Yang-Baxter equation, using an arbitrary skew lattice. In this section, we study another map $r$, defined by $r(x,y)=(x \land y, x \lor y)$, for all $x,y$ elements of a given skew lattice $(S, \land, \lor)$. 

The inspiration for the definition of this map $r$ comes from the following well-known result. To each lattice $(L,\land,\lor)$ one may associate an idempotent map $r:L\times L \rightarrow L\times L$ defined as
\begin{equation}\label{eq:defR}
r(x,y)=(x \land y, x \lor y). 
\end{equation} Moreover, $(L,r)$ is a solution of the Yang-Baxter equation, i.e.
\begin{equation}\label{eq:yb}
(r \times \text{id}) \circ (\text{id} \times r) \circ (r \times \text{id}) = (\text{id} \times r) \circ (r \times \text{id}) \circ (\text{id} \times r),
\end{equation}
if and only if $L$ is a distributive lattice, meaning that for any $x,y,z \in L$, $x\land(y\lor z) =(x\land y)\lor(x\land z)$ and $x\lor(y\land z) =(x\lor y)\land(x\lor z)$.

We say that a skew lattice $S$ is a \emph{strong distributive solution} of the Yang-Baxter equation, if the  map $r:S \times S \rightarrow S \times S$ defined by
\eqref{eq:defR}
is a set-theoretic solution of the Yang-Baxter equation \eqref{eq:yb}. 
%


\begin{theorem}\label{th:str-sol-var}
	The class of skew lattices that form strong distributive solutions of the Yang-Baxter equation is a variety. Moreover, this variety is defined by the following identities
	\begin{eqnarray}
	\label{eq:identity-strong1} x\land y\land ((x\lor y)\land z) &=& x\land y\land z, \\ \label{eq:identity-strong2}
	(x\land y) \lor ((x\lor y)\land z) &=& (x\lor (y\land z))\land (y\lor z), \\ \label{eq:identity-strong3}
	x\lor y\lor z &=& x\lor (y\land z)\lor y\lor z.
	\end{eqnarray}
\end{theorem}

\begin{proof}
	Let $x,y,z \in S$. Then
	\begin{align*}
	&(r\times \id) (\id\times r)(r\times \id) (x, y,z) \\&=(r\times \id) (\id\times r)(x\land y,x\lor y, z) 
	\\ &= (r\times \id)(x \land y, (x \lor y) \land z, (x \lor y) \lor z)
	\\ &= ((x \land y)\land ((x \lor y) \land z), (x \land y)\lor ((x \lor y) \land z), (x \lor y) \lor z),
	\end{align*}
	and
	\begin{align*}
	&(\id\times r)(r\times \id)(\id\times r) (x, y,z) \\&=(\id\times r)(r\times \id)(x, y\land z, y \lor z)
	\\ &= (\id\times r)(x \land (y\land z), x \lor (y\land z), y \lor z)
	\\ &= (x \land (y\land z), (x \lor (y\land z))\land (y \lor z),(x \lor (y\land z)) \lor (y \lor z)).
	\end{align*}
	Hence a skew lattice $S$ is a strong distributive solution of the Yang-Baxter equation if and only if it satisfies
	\begin{eqnarray*}
		(x \land y)\land ((x \lor y) \land z) &=& x \land (y\land z),\\
		(x \land y)\lor ((x \lor y) \land z) &=& (x \lor (y\land z))\land (y \lor z),\\
		(x \lor y) \lor z &=& (x \lor (y\land z)) \lor (y \lor z).
	\end{eqnarray*}
\end{proof}

In Section 2, we mentioned that there are several ways to generalize the notion of distributivity for lattices. The weakest of them is the notion of \emph{quasi-distributivity}\label{quasi-distributive} which is defined as having a distributive maximal lattice image. In particular it means that its maximal lattice image $S/\DD$ is a distributive lattice, \cite{L1}. If $S$ is a distributive skew lattice, then it is always {quasi-distributive}. A thorough study of distributivity in skew lattices can be found in \cite{dist}. 
For lattices, the notion of cancellation is equivalent to distributivity. As a consequence, cancellative skew lattices are always quasi-distributive.  

\begin{corollary}
	Let $S$ be a skew lattice. If $S$ is a strong distributive solution of the Yang-Baxter equation \eqref{eq:yb}, then
	\begin{itemize}
		\item[(i)] The maximal lattice image $S/\DD$ is also a strong distributive solution of \eqref{eq:yb}.
		\item[(ii)] $S$ is quasi-distributive.
	\end{itemize}
\end{corollary}

\begin{proof}
	(i) By Theorem \ref{th:str-sol-var} strong distributive solutions of \eqref{eq:yb} form a variety. Therefore all homomorphic images of strong distributive solutions are again strong distributive solutions. The maximal lattice image $S/\DD$ is the homomorphic image of $S$ under the natural projection $\pi:S\to S/\DD$ which maps each element to its $\DD$-class.
	
	(ii) By definition, $S$ is quasi-distributive if and only if $S/\DD$ is distributive. By (i), the lattice $S/\DD$ is a strong distributive solution.  But then $S/\DD$ must be a distributive lattice since lattices are distributive if and only if they are strong distributive solutions if and only if the map (\ref{eq:defR}) satisfies equation (\ref{eq:yb}). 
\end{proof}

In order to state and prove our next theorem, we need to introduce several varieties of skew lattices. 

A skew lattice is said to be \emph{normal}\label{normal} if it satisfies the identity
\[
x\land y\land z\land x = x\land z\land y\land x.
\]
It is said to be \emph{conormal}\label{conormal} if it satisfies the identity
\[
x\lor y\lor z\lor x = x\lor z\lor y\lor x.
\]
It is an easy exercise to prove that the condition of normality (resp. conormality) is equivalent to  $x \land y \land z \land w = x \land z \land y \land w$ (resp. $x \lor y \lor z \lor w = x \lor z \lor y \lor w$).
A skew lattice is called \emph{binormal} if it is both normal and conormal. By a result of Schein \cite{Schein}, a binormal skew lattice factors as a direct product of a lattice with a rectangular algebra. 

A skew lattice is said to be \emph{strongly distributive}\label{strongly distributive} if it satisfies the identities
\[
(x\lor y)\land z = (x\land z)\lor (y\land z) \text{ and } x\land (y\lor z)=(x\land y)\lor (x\land z).
\]

A skew lattice is said to be \emph{co-strongly distributive}\label{co-strongly distributive} if it satisfies the identities
\[
(x\land y)\lor z = (x\lor z)\land (y\lor z) \text{ and } x\lor (y\land z)=(x\lor y)\land (x\lor z).
\]

By a result of Leech \cite{L4}, a skew lattice that is either strongly distributive or co-strongly distributive is distributive. Leech \cite{L4} proved that a skew lattice $S$ is strongly distributive if and only if it is symmetric, quasi-distributive and normal. Dually, $S$ is co-strongly distributive if and only if it is symmetric, quasi-distributive and conormal.
Moreover, by a result in \cite{dist}, cancellation is implied either by strong distributivity or co-strong distributivity.

\begin{theorem}\label{th:solution-strong}
	Let $(S,\land,\lor)$ be a skew lattice which is both strongly and co-strongly distributive. Then $S$ is a strong distributive solution of the Yang-Baxter equation. Furthermore, this solution is cubic, i.e. $r^3=r$.
\end{theorem}

\begin{proof}
	By Theorem \ref{th:str-sol-var} we need to prove that $S$ satisfies the identities (\ref{eq:identity-strong1})--(\ref{eq:identity-strong3}).
	Let $x,y,z \in S$. 
	
	Recall that strong distributivity implies normality, and co-strong distributivity implies conormality.
	Using strong distributivity and normality, we deduce
	\begin{align*}
	&(x \land y)\land ((x \lor y) \land z) 
	\\ &= ((x \land y \land x) \lor (x \land y)) \land z
	\\ &= (x \land y \land x \land z) \lor (x \land y \land z)
	\\ &= (x \land x \land y \land z) \lor (x \land y \land z)
	\\ &= (x \land y \land z) \lor (x \land y \land z)
	\\ &= x \land (y \land z).
	\end{align*}
	
	Using co-strongly distributivity and conormality, we obtain
	\begin{align*}
	&(x \lor (y\land z)) \lor (y \lor z) 
	\\ &= x \lor ((y \lor z) \land (z \lor y \lor z))
	\\ &= (x \lor y \lor z) \land (x \lor z \lor y \lor z)
	\\ &= (x \lor y \lor z) \land (x \lor y \lor z \lor z)
	\\ &= (x \lor y \lor z) \land (x \lor y \lor z)
	\\ &= (x \lor y) \lor z.
	\end{align*}
	Thus, we are left to prove that
	\begin{equation}\label{eq1}
	(x \land y)\lor ((x \lor y) \land z) = (x \lor (y\land z))\land (y \lor z).
	\end{equation}
	The left hand side of this equation is equal to
	\begin{align*}
	(x \land y)\lor ((x \lor y) \land z)
	= (x \land y) \lor (x \land z) \lor (y \land z),
	\end{align*}
	where we used that the skew lattice is strongly distributive.
	The right hand side of equation (\ref{eq1}) can be rewritten as follows, using strong distributivity, conormality and the absorption rules, 
	\begin{align*}
	&(x \lor (y\land z))\land (y \lor z)
	\\ &= (x \land (y \lor z)) \lor ((y \land z) \land (y \lor z))
	\\ &= (x \land y) \lor (x \land z) \lor (y \land z \land y) \lor (y \land z)
	\\ &= (x \land y) \lor (x \land z) \lor (y \land z \land y) \lor (y \land z) \lor (y \land z)
	\\ &= (x \land y) \lor (x \land z) \lor (y \land z) \lor (y \land z \land y) \lor (y \land z)
	\\ &= (x \land y) \lor (x \land z) \lor (y \land z) \lor (y \land z)
	\\ &= (x \land y) \lor (x \land z) \lor (y \land z).
	\end{align*}
	
	Let $x,y \in S$. Then,
	\begin{align*}
	r^3(x,y) 
	= &\; r^2(x \land y, x \lor y)
	\\ = &\; r((x \land y) \land (x \lor y), (x \land y) \lor (x \lor y))
	\\ = &\; (((x \land y) \land (x \lor y)) \land ((x \land y) \lor (x \lor y)), \\ &\; ((x \land y) \land (x \lor y)) \lor ((x \land y) \lor (x \lor y))).
	\end{align*}
	Using normality and the absorption rule, we deduce
	\begin{align*}
	&((x \land y) \land (x \lor y)) \land ((x \land y) \lor (x \lor y)) 
	\\ &= (x \land y) \land (x \lor y) \land (x \land y) \land ((x \land y) \lor (x \lor y)) 
	\\ &= (x \land y) \land (x \lor y) \land (x \land y)
	\\ &= x \land y \land x \land (x \lor y) \land x \land y
	\\ &= x \land y \land x \land x \land y
	\\ &= x \land y.
	\end{align*}
	Similarly, using conormality and the absorption rule, 
	\begin{align*}
	&((x \land y) \land (x \lor y)) \lor ((x \land y) \lor (x \lor y)) 
	\\&= ((x \land y) \land (x \lor y)) \lor (x \lor y)  \lor (x \land y) \lor (x \lor y) 
	\\&= (x \lor y)  \lor (x \land y) \lor (x \lor y)
	\\&= x \lor y \lor x \lor (x \land y) \lor x \lor y
	\\&= x \lor y \lor x \lor x \lor y
	\\&= x \lor y.
	\end{align*}
	Hence, $r^3(x,y) = (x\land y, x \lor y) = r(x,y)$.
\end{proof}

In general, we cannot omit either strong distributivity or co-strong distributivity from the assumptions of Theorem \ref{th:solution-strong}, as is verified by the following pair of examples.

\begin{example}\label{ex:3R0} \normalfont 
	Let $\mathbf 3^{R,0}$ be a $3$-element skew lattice given by the following pair of Cayley tables:
	\[
	\begin{tabular}{r|rrr}
	$\land$ & 0 & 1 & 2\\
	\hline
	0 & 0 & 0 & 0 \\
	1 & 0 & 1 & 2 \\
	2 & 0 & 1 & 2
	\end{tabular} \hspace{.5cm}
	\begin{tabular}{r|rrr}
	$\lor $ & 0 & 1 & 2\\
	\hline
	0 & 0 & 1 & 2 \\
	1 & 1 & 1 & 1 \\
	2 & 2 & 2 & 2
	\end{tabular}
	\]
	
	It is easy to check that $\mathbf 3^{R,0}$ is a right handed skew lattice with two comparable $\DD$-classes $\{1,2\}>\{0\}$, and it is strongly distributive (but not co-strongly distributive) by a result of Leech \cite{L4}. We claim that $\mathbf 3^{R,0}$ is not a strong distributive solution, more specifically, it does not satisfy the identity \eqref{eq:identity-strong2}. Take $x=0$, $y=1$ and $z=2$. Then $(x\land y)\lor ((x\lor y)\land z)=(0\land 1)\lor ((0 \lor 1)\land 2)=0\lor (1\land 2)=0\lor 2=2$, while $(x\lor (y\land z))\land (y\lor z)=(0\lor (1\land 2))\land (1\lor 2)=(0\lor 2)\land 1=2\land 1=1$.
\end{example}

\begin{example}\label{ex:3R1} \normalfont 
	Let $\mathbf 3^{R,1}$ be a $3$-element skew lattice given by the following pair of Cayley tables:
	\[
	\begin{tabular}{r|rrr}
	$\land$ & 0 & 1 & 2\\
	\hline
	0 & 0 & 0 & 2 \\
	1 & 0 & 1 & 2 \\
	2 & 0 & 2 & 2
	\end{tabular} \hspace{.5cm}
	\begin{tabular}{r|rrr}
	$\lor$ & 0 & 1 & 2\\
	\hline
	0 & 0 & 1 & 0 \\
	1 & 1 & 1 & 1 \\
	2 & 2 & 1 & 2
	\end{tabular}
	\]
	Similar argumentation as in Example \ref{ex:3R0} shows that $\mathbf 3^{R,1}$ is a co-strongly distributive  (but not strongly distributive) skew lattice, which is not a strong distributive solution.
\end{example}

More can be said in the case of left handed skew lattices.

\begin{proposition}
	Let $S$ be a left handed skew lattice. Then $S$ satisfies the identities \eqref{eq:identity-strong1} and \eqref{eq:identity-strong3}. If  in addition to being left handed, $S$ is  either strongly distributive or co-strongly distributive, then it also satisfies \eqref{eq:identity-strong2} and is thus a strong distributive solution.
\end{proposition}

\begin{proof}
	We first show that $S$ satisfies \eqref{eq:identity-strong1}. Using left handedness we obtain $(x\land y)\land ((x\lor y)\land z)=x\land y\land x\land (x\lor y)\land z$, which by absorption simplifies to $x\land y\land x\land z$, and then by left handedness further to $x\land y\land z$.
	
	Next we show that $S$ satisfies \eqref{eq:identity-strong3}. Using left handedness we obtain $x\lor (y\land z)\lor y\lor z=x\lor (y\land z)\lor z\lor y\lor z$. With the aid of absorption and left handedness, the latter first simplifies to $x\lor z\lor y\lor z$, and then to $x\lor y\lor z$.
	
	Assume now that $S$ is strongly distributive. (The case when $S$ is co-strongly distributive is handled in a dual fashion.) We claim that  $S$ satisfies \eqref{eq:identity-strong2}. We obtain $(x\land y)\lor ((x\lor y)\land z)=(x\land y)\lor (x\land z)\lor (y\land z)$, and $(x\lor (y\land z))\land (y\lor z)=(x\land (y\lor z))\lor (y\land z\land (y\lor z))$. Using strong distributivity and left handedness this expands to $(x\land y)\lor (x\land z)\lor (y\land z\land y\land (y\lor z))$. Using absorption and left handedness this simplifies to $(x\land y)\lor (x\land z)\lor (y\land z)$. 
\end{proof}

It turns out that not all left handed, distributive and cancellative skew lattices are  strong distributive solutions. The program \emph{Mace4} \cite{McCune} was able to find a $16$-element example of a left handed, distributive and cancellative skew lattice, that is not a strong distributive solution. 

The set-theoretic solution (\ref{eq:defR}) obtained from a strong distributive skew lattice is degenerate in general. Nevertheless, there are examples where the solution is non-degenerate.

\begin{example}\label{ex1} \normalfont 
	Let $S$ be a non-empty set and let the skew lattice operations be defined on $S$ by $x \land y = y$ and $x \lor y = x$, for all $x,y \in S$. Then, one can check that this is a strongly and co-strongly distributive skew lattice, and thus a strong distributive solution by Theorem \ref{th:solution-strong}. In fact, $x\,\RR\, y$ holds for all $x,y\in S$, and thus $(S, \land)$ is a \emph{right-zero semigroup} (i.e. it satisfies $x\land y=y$), which means that $S$ is a right handed skew lattice with  one $\DD$-class. The associated map (\ref{eq:defR}) is the twist map $r(x,y)=(y,x)$. This solution is non-degenerate as both $\lambda_x: X \rightarrow X: t \mapsto x \land t = t$ and $\rho_y: X \rightarrow X: t \mapsto t \lor y=t$ are bijective maps, for all $x,y \in S$.
\end{example}

In fact, the skew lattices given by Example \ref{ex1} above are the only strong distributive solutions that give non-degenerate set-theoretic solutions of the Yang-Baxter equation.

\begin{proposition}\label{prop:strong-non-deg}
	Let $(S,\land,\lor)$ be a skew lattice that is a strong distributive solution, where the associated solution is left or right non-degenerate. Then, $(S, \land, \lor)$ is the skew lattice from Example \ref{ex1}.
\end{proposition}

\begin{proof}
	Assume first that $S$ is a strong distributive solution and that the obtained solution from the map (\ref{eq:defR}) is left non-degenerate.
	Let $x,y\in S$ be arbitrary. We claim that $x \land y = y$ and $x \lor y = x$. By the assumption, the map $\lambda_x:t\mapsto x\land t$, is a bijection. So, there exists $t \in S$ such that $y = x \land t$, and thus $x \land y = x \land (x \land t)= x \land t=y$. Furthermore, using absorption, $x \lor y = x \lor (x \land y) = x$. Thus, we obtain a skew lattice as in Example \ref{ex1}.
	
	The right non-degenerate case is similar to the left non-degenerate case.
	
\end{proof}

\section{More distributive solutions}\label{sec:sd}

In this section we show that to skew lattices one can naturally associate more idempotent solutions. To do so we need the following terminology.

\subsection{Left distributive solutions}
Let $S$ be a skew lattice. Consider the map $r_L:S\times S \rightarrow S \times S$ defined by
\begin{equation}\label{eq:defrl}
r_L(x,y)=(x\land y, y\lor x).
\end{equation}
We say that a skew lattice $S$ is a \emph{left distributive solution} of the Yang-Baxter equation, if $(S, r_L)$ is a set-theoretic solution of the Yang-Baxter equation \eqref{eq:yb}.

\begin{proposition}
	Let $S$ be a skew lattice. If $S$ is a left distributive solution of the Yang-Baxter equation, then $(S,r_L)$ is an idempotent solution.
\end{proposition}

\begin{proof}
	For any $x,y \in S$, using the absorption laws we obtain 
	\begin{align*}
	r_L^2(x,y) &= r_L(x\land y, y \lor x) \\ &= ((x\land y) \land (y \lor x), (y \lor x) \lor (x \land y)) \\ &=  (x\land y, y \lor x) \\&= r_L(x,y).
	\end{align*}
\end{proof}

\begin{theorem}\label{th:left-sol-var}
	The class of left distributive solutions of the Yang-Baxter equation is a variety. Moreover, this variety is defined by the identity
	\begin{equation}\label{eq:identity-left}
	((y\lor x)\land z)\lor (x\land y)=((y\land z)\lor x)\land (z\lor y).
	\end{equation}
\end{theorem}

\begin{proof}
	Denote $r=r_L$. A skew lattice $S$ is a left distributive solution of the Yang-Baxter equation if and only if it satisfies
	\begin{equation}\label{eq:ybe}
	(r\times \id) (\id\times r)(r\times \id) (x, y,z)= (\id\times r)(r\times \id) (\id\times r)(x, y,z).
	\end{equation}
	Computing the left side of \eqref{eq:ybe} yields
	\begin{align*}
	&(r\times \id) (\id\times r)(r\times \id) (x, y,z) \\ &=(r\times \id) (\id\times r)(x\land y,y\lor x, z) 
	\\ &= (r\times \id)(x\land y, (y\lor x)\land z, z\lor y\lor x) 
	\\ &= (x\land y\land (y\lor x)\land z), ((y\lor x)\land z)\lor (x\land y), z\lor y\lor x).
	\end{align*}
	On the other hand, the right side of \eqref{eq:ybe} expands as
	\begin{align*}
	&(\id\times r)(r\times \id) (\id\times r)(x, y,z)
	\\ &=(\id\times r)(r\times \id) (x, y\land z ,z\lor y) 
	\\ &=
	(\id\times r)(x\land y\land z, (y\land z)\lor x, z\lor y) 
	\\ &= (x\land y\land z, ((y\land z)\lor x)\land (z\lor y), z\lor y\lor (y\land z)\lor x).
	\end{align*}
	By absorption, $x\land y\land (y\lor x)\land z$ reduces to $x\land y\land z$. Similarly, $z\lor y\lor (y\land z)\lor x$ reduces to $z\lor y\lor x$. Hence, $S$ is a left distributive solution if and only it satisfies the identity $((y\lor x)\land z)\lor (x\land y)=((y\land z)\lor x)\land (z\lor y)$.
\end{proof}

\begin{corollary}\label{cor:maximal-lat-im-dist-sol}
	Let $S$ be a skew lattice. If $S$ is a left distributive solution of the Yang-Baxter equation, then
	the maximal lattice image $S/\DD$ is also a left distributive solution,
	and thus $S$ is quasi-distributive.
\end{corollary}

In the previous section we obtained that strongly and co-strongly distributive skew lattices are strong distributive solution. The following result shows that these skew lattices are also left distributive solutions. Recall that a skew lattice is called strongly distributive if it satisfies the identities $(x\lor y)\land z = (x\land z)\lor (y\land z)$ and $x\land (y\lor z)=(x\land y)\lor (x\land z)$ and co-strongly distributive if it satisfies the identities $(x\land y)\lor z = (x\lor z)\land (y\lor z)$ and $x\lor (y\land z)=(x\lor y)\land (x\lor z)$.

\begin{proposition}\label{th:solution-sd}
	Let $S$ be a skew lattice that is strongly distributive or co-strongly distributive. Then $S$ is a left distributive solution of the Yang-Baxter equation. 
\end{proposition}

\begin{proof}
	We give a proof for the case of strongly distributive skew lattices. The case of co-strongly distributive skew lattices is handled in a dual fashion. By Theorem \ref{th:left-sol-var}, we need to prove that $S$ satisfies the identity $((y\lor x)\land z)\lor (x\land y)=((y\land z)\lor x)\land (z\lor y)$. Let $x,y,z\in S$ be arbitrary. Using strong distributivity  $((y\lor x)\land z)\lor (x\land y)$
	simplifies to
	$(y\land z)\lor (x\land z)\lor (x\land y)$.
	On the other hand,  $((y\land z)\lor x)\land (z\lor y)$ simplifies to $(y\land z\land (z\lor y))\lor (x\land(z\lor y))$, which using absorption and strong distributivity further simplifies to  $(y\land z)\lor (x\land z)\lor (x\land y)$.
\end{proof}

In fact, the result of Proposition \ref{th:solution-sd} can be strengthened to a more general class of  skew lattices. 

To prove Lemma \ref{lemma:coset-law} and Proposition \ref{th:solution-dist-canc-left} below we use the technique of Remark \ref{remark1}. 
Furthermore, by \cite[Proposition 7]{Cv2011}, elements $a,a'$ of $A$ lie in a common coset of $B$ in $A$ if and only if $b\lor a\lor b=b\lor a'\lor b$ for all $b\in B$, which is further equivalent to $b\lor a\lor b=b\lor a'\lor b$ for some $b\in B$. Dually, elements $b,b'\in B$ lie in a common coset of $A$ in $B$ if and only if $a\land b\land a=a\land b'\land a$ for all $a\in A$, which is further equivalent to $a\land b\land a=a\land b'\land a$ for some $a\in A$. 

A \emph{skew diamond} $\{J>A,B>M\}$ is a sub-skew lattice of a skew lattice $S$ with four $\DD$-classes $A,B,M,J$, such that $M=A\land B$ and $J=A\lor B$. Given a skew diamond $\{J>A,B>M\}$, the cosets of $A$ in $J$ are given by $A\lor b\lor A$, where $b\in B$. Likewise, the cosets of $A$ in $M$ are given by $A\land b\land A$, where $b\in B$. 

Finally, we define some more varieties of skew lattices that are used in the following results. 
Recall that a skew lattice is called \emph{symmetric} if for any $x,y\in S$, $x\land y=y\land x$ if and only if $x\lor y=y\lor x.$ 
Moreover, a skew lattice is said to be \emph{upper symmetric} if $x\land y=y\land x$ implies $x\lor y=y\lor x$; and it is called \emph{lower symmetric} if $x\lor y=y\lor x$ implies $x\land y=y\land x$. Finally, a skew lattice is called \emph{simply cancellative}\label{simply cancellative} if it satisfies the following implication
\begin{eqnarray*}
	z\lor x \lor z = z\lor  y\lor z, z\land x \land z =z\land y \land z \Longrightarrow x = y.
\end{eqnarray*}

A skew lattice is cancellative if and only if it is simply cancellative and symmetric \cite{canc}.

\begin{lemma}\label{lemma:coset-law}
	Let $ S$ be a simply cancellative skew lattice, $\{J>A,B>M\}$  a skew diamond in $ S$ and  $x_1,x_2\in A$.
	\begin{itemize}
		\item[(i)] Let $S$ be upper symmetric. If $ B \lor x_1 \lor B= B \lor x_2\lor B$, then $M\lor x_1\lor M= M\lor x_2\lor M$.
		\item[(ii)] Let $S$ be lower symmetric. If $B\land x_1 \land B=B\land x_2 \land B$, then $J\land x_1 \land J=J\land x_2 \land J$.
	\end{itemize}
\end{lemma}

\begin{proof}
	(i). Let $S$ be upper symmetric and
	$B \lor x_1 \lor B= B \lor x_2\lor B$. 
	Assume that $M\lor x_1\lor M\neq M\lor x_2\lor M$, so there exists $m\in M$ such that $a_1\neq a_2$, where $a_1=m \vee x_1\vee m$ and $a_2=m \vee x_2\vee m$. 
	Note that $a_1,a_2\in A$, $m<a_1$, $m<a_2$. 
	Take  $b\in B$ such that $m<b$.  Since $m<a_1$ and $m<b$ it follows that $a_1\land b=m=b\land a_1$, and likewise $a_2\land b=m=b\land a_2$. Because $S$ is upper symmetric, we obtain  $a_1\lor b=b\lor a_1$  and $a_2\lor b=b\lor a_2$. Denote $j_1=a_1\lor b$ and $j_2=a_2\lor b$. The assumption  $B \lor x_1 \lor B= B \lor x_2\lor B$ implies $b\vee a_1 \vee b=b\vee x_1\vee b=b\vee x_2\vee b=b\vee a_2\vee b$. 
	It follows that $j_1=j_2$, and the set
	$\{m,a_1,a_2, b,j_1\}$ forms a subalgebra in $S'$, given by the following diagram:
	
	
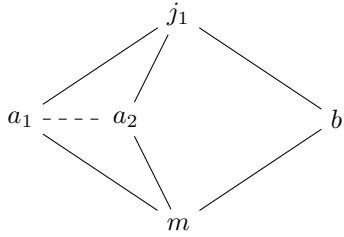
\[
	\begin{tikzpicture}[scale=.7]
	\node (1) at (0,2){$j_1$} ;
	\node (a) at (-3,0){$a_1$} ;
	\node (b) at (-1,0){$a_2$};
	\node (c) at (3,0){$b$}  ;
	\node (0) at (0,-2){$m$} ;
	\draw (1) -- (c) -- (0) -- (a) -- (1) -- (b) -- (0);
	\draw[dashed] (a) -- (b);
	\end{tikzpicture}
	\]

	The subalgebra $S'$ is isomorphic either to $\mathbf{NC}_5^{\RR}$ (a right handed skew lattice in which $a_1\land a_2=a_2$, $a_2\land a_1=a_1$, $a_1\lor a_2=a_1$ and $a_2\lor a_1=a_2$) or to $\mathbf{NC}_5^{\LL}$ (a left handed skew lattice in which $a_1\land a_2=a_1$, $a_2\land a_1=a_2$, $a_1\lor a_2=a_2$ and $a_2\lor a_1=a_1$). It was proven in \cite{canc} that a skew lattice is  simply cancellative if and only if it contains no sub-skew lattice isomorphic to $\mathbf{NC}_5^{\RR}$ or $\mathbf{NC}_5^{\LL}$. Thus $S$ is not simply cancellative, which is a contradiction.

	The proof of (ii) is similar. 
\end{proof}

\begin{proposition}\label{th:solution-dist-canc-left}
	Let $S$ be a left (respectively right) handed, distributive, simply cancellative and lower (respectively upper) symmetric skew lattice. Then $S$ is a left distributive solution of the Yang-Baxter equation. 
\end{proposition}

\begin{proof}
	Let $x,y,z\in S$ be arbitrary and denote $\alpha = ((y\lor x)\land z)\lor (x\land y)$, $\beta=((y\land z)\lor x)\land (z\lor y)$. Denote the corresponding $\DD$-classes by  $X=\DD_x$, $Y=\DD_y$, $Z=\DD_z$,  $M=\DD_\alpha$. By Theorem \ref{th:left-sol-var} we need to prove that $\alpha=\beta$. Since $S$ is distributive it follows that $S/\DD$ is a distributive lattice, and thus a left distributive solution by Corollary \ref{cor:maximal-lat-im-dist-sol}, i.e. $\alpha = \beta$ in $S/\DD$. Hence, $\alpha \,\DD\, \beta$. We divide the proof into several steps.
	
	First, we consider the following skew diamond:
	
	
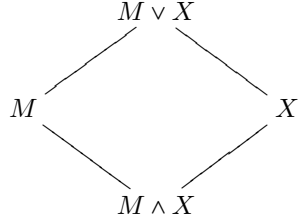
\[
	{\xymatrix @H=1pt {
			& M\lor X                                    &                       \\
			M\ar@{-}[ur]   &                 & X \ar@{-}[ul]   \\
			& M\land X  \ar@{-}[ur]\ar@{-}[ul]&
		}}
		\]
		
		(a) We claim that $X\land \alpha \land X=X\land \beta \land X$, i.e. $x\land \alpha \land x$ and $x\land \beta \land x$ lie in the same coset of $X$ in $M\land X$. Since the cosets of $X$ form a partition of $M\land X$, it suffices to prove $x\land \alpha \land x=x\land \beta \land x.$ Using distributivity (\ref{eq:D1}), regularity (\ref{eq:reg1}) and absorption we obtain
		\begin{align*}
		x\land \alpha\land x & = (x\land ((y\lor x)\land z)\land x)\lor (x\land (x \land y)\land x) \\
		& = (x\land (y\lor x)\land x\land z\land x)\lor (x\land y\land x) \\
		& =  (x\land z\land x)\lor (x\land y\land x) \\
		& =  x\land (z\lor y)\land x.
		\end{align*}
		On the other hand, using regularity (\ref{eq:reg1}) and absorption we obtain
		\begin{align*}
		x\land \beta \land x &= x\land ((y\land z)\lor x) \land x\land (z\lor y) \land x \\
		& = x\land (z\lor y)\land x.
		\end{align*}

		(b) We claim that $(M\lor X)\land \alpha \land (M\lor X)=(M\lor X)\land \beta\land (M\lor X)$.  By Lemma \ref{lemma:coset-law}, given any skew diamond $\{J>A,B>M\}$ in a simply cancellative and lower symmetric skew lattice $S$ and any $a,a'\in A$, $B\land a\land B=B\land a'\land B$ implies $J\land a\land J=J\land a'\land J$. Applying this to the above diagram ($X$ in the role of $B$, $M\lor X$ in the role of $J$, $\alpha$ in the role of $a$, and $\beta$ in the role of $b'$),  (a) yields exactly $(M\lor X)\land \alpha \land (M\lor X)=(M\lor X)\land \beta\land (M\lor X)$.
		
		Denote further $A=M\lor X$, $B=M\lor Y$ and $J=A\lor B$. Note that $M=(X\land Y)\lor (X\land Z)\lor (Y\land Z)$, $A\land B=M$ and $J=X\lor Y$. We consider the  skew diamond below:
		
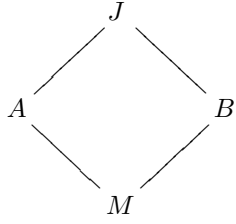
\[
		{\xymatrix @H=1pt {
				& J\                                   &                       \\
				A \ar@{-}[ur]   &                 & B \ar@{-}[ul]   \\
				& M  \ar@{-}[ur]\ar@{-}[ul]&
			}}
			\]
			
			(c) We claim that $J\land \alpha\land J=J\land \beta\land J$. Observe that using our new notation, we have just proved that $\alpha$ and $\beta$ lie in the same coset of $A$ in $M$. Similarly, we can prove that they lie in the same coset of $B$ in $M$. (In order to prove this we first need to repeat step (a) above, using $Y$ instead of $X$ and show $y\land \alpha \land y=y\land \beta\land y$.) By a result of \cite{L6}, given any skew diamond in a lower symmetric skew lattice, cosets of $J$ in $M$ are exactly intersections of cosets of $A$ in $M$ by cosets of $B$ in $M$. It follows that $\alpha$ and $\beta$ lie in the same coset of $J$ in $M$.
			
			(d)  We have just proven that $\alpha $ and $\beta $ lie in the same coset of $J$ in $M$. In order to prove that their are equal it thus suffices to show that both lie below a common element of $J$. In fact, we claim that $\alpha\leq y\lor x$ and $\beta \leq y\lor x$.  
			Using absorption we obtain $y\lor x\lor \alpha = y\lor x\lor ((y\lor x)\land z)\lor (x\land y)=y\lor x\lor (x\land y)=y\lor x$. On the other hand, since $S$ is left handed, we obtain $\alpha \lor (y\lor x)=(y\lor x)\lor \alpha\lor (y\lor x)=(y\lor x)\lor (y\lor x)=y\lor x$, and thus $\alpha \leq y\lor x$.
			
			Moreover, since $S$ is left handed, we obtain $\beta=(x\lor (y\land z)\lor x)\land (z\lor y)$, which by distributivity expands to $(x\lor y\lor x)\land (x\lor z\lor x)\land (z\lor y)$, and then by left handedness to $(y\lor x)\land (z\lor x)\land (z\lor y)$. It follows that $(y\lor x)\land \beta =\beta$, and $\beta\land (y\lor x)=\beta\land (y\lor x)\land \beta=\beta\land \beta=\beta$. Thus $\beta\leq y\lor x.$
			
			Elements $\alpha$, $\beta$ lie in the same coset of $J$ in $M$ and are below the same element of $J$, thus they must be equal.
			
			The case where $S$ is a right handed, distributive, simply cancellative and upper symmetric skew lattice is handled similarly.
		\end{proof}

		The following theorem characterizes a left distributive solution in terms of varieties of a skew lattice. We define a \emph{left cancellative} skew lattice as a skew lattice satisfying implication (\ref{C1}):
		\begin{align*}
		x \lor y = x \lor z, x \land y = x \land z \Longrightarrow y = z. 
		\end{align*}
		
		\begin{theorem}\label{th:solution-dist-canc}
			Let $S$ be a skew lattice. Then, $S$ is distributive and  left cancellative if and only if $S$ is a left distributive solution of the Yang-Baxter equation. 
		\end{theorem}
		
		\begin{proof}
			Let $S$ be any distributive and left cancellative skew lattice. The left factor $S_L$ of $S$ is a left handed and left cancellative skew lattice, and thus  it is simply cancellative and lower symmetric  by  \cite[Theorem 5.1]{canc}. Since $S_L$ is also a distributive skew lattice, it  follows  by Proposition \ref{th:solution-dist-canc-left} that it is a left distributive solution. Dually, the right factor $S_R$ of $S$ is a right handed, distributive and left cancellative skew lattice, and thus  it is simply cancellative and upper symmetric  by a result of \cite[Theorem 5.1]{canc}. It follows by Proposition \ref{th:solution-dist-canc-left} that $S_R$ is also a left distributive solution. By Theorem \ref{thm:lrhanded}, $S $ is a left distributive solution. That proves the direct implictation.
			
			The converse was proven by the Automated Theorem Prover \emph{Prover9} \cite{McCune}, which was able to derive a proof that every left distributive solution is distributive and left cancellative.
		\end{proof}

		Similar to strong distributive solutions, the set-theoretic solution obtained from a left distributive solution will be degenerate in general. Nevertheless, there are examples where the solution is left non-degenerate. Take for instance the skew lattice from Example \ref{ex1}. This skew lattice $S$ is a left distributive solution and one can see that $r_L(x,y) = (y,y)$, for all $x,y \in S$. Hence, we obtain a left non-degenerate solution.

		\subsection{Right distributive solutions}
		Let $S$ be a skew lattice. Consider the map $r_R:S\times S \rightarrow S \times S$ defined by
		\begin{equation}\label{eq:defrr}
		r_R(x,y)=(y\land x, x\lor y).
		\end{equation}
		We say that a skew lattice $S$ is a \emph{right distributive solution} of the Yang-Baxter equation, if $(S, r_R)$ is a set-theoretic solution of the Yang-Baxter equation \eqref{eq:yb}. Note that $r_R = r_L \circ r'$, where $r'$ is the twist map $r'(x,y)=(y,x)$.
		
		The following theorem is proven in a similar fashion as the corresponding results for left distributive solutions. The latter used left cancellative skew lattices. So for the next theorem we need the definition of a \emph{right cancellative} skew lattice, which is a skew lattice satisfying implication (\ref{C2}):
		\begin{align*}
		x \lor z = y \lor z, x \land z = y \land z \Longrightarrow x = y. 
		\end{align*}

		\begin{theorem}\label{th:right-sol-var}
			\begin{itemize}
				\item[(i)] The class of right distributive solutions of the Yang-Baxter equation is a variety. Moreover, this variety is defined by the identity
				\begin{equation}\label{eq:identity-right}
				(y\land x)\lor (z\land (x\lor y))=(y\lor z)\land (x\lor (z\land y)).
				\end{equation}
				\item[(ii)] Right distributive solutions are always idempotent, i.e. $r_R^2=r_R$.
				\item[(iii)] Every strong distributive solution is also a right distributive solution.
				\item[(iv)] Every left handed, distributive, simply cancellative and upper symmetric skew lattice is a right distributive solution.
				\item[(v)] Every right handed, distributive, simply cancellative and lower symmetric skew lattice is a right distributive solution.
				\item[(vi)] Every distributive and right cancellative skew lattice is a right distributive solution.
				\item[(vii)] [Modulo \emph{Prover9}] A skew lattice is a right distributive solution if and only if it is distributive and right cancellative.
			\end{itemize}
		\end{theorem}

		Similar to strong distributive solutions, the set-theoretic solution obtained from a right distributive solution will be degenerate in general. Nevertheless, there are examples where the solution is right non-degenerate, and thus not degenerate. Take again the skew lattice from Example \ref{ex1}. This skew lattice $S$ is a right distributive solution and one can see that $r_R(x,y)=(x,x)$, for all $x,y \in S$. Hence,  $(S,r_R)$ is a right non-degenerate solution.
		
		\subsection{Weak distributive solutions}
		
		Let $S$ be a skew lattice. Consider the map $r_W:S\times S \rightarrow S \times S$ defined by
		\begin{equation}\label{eq:defrw}
		r_W(x,y)=(x\land y\land x, x\lor y\lor x).
		\end{equation}
		We say that a skew lattice $S$ is a \emph{weak distributive solution} of the Yang-Baxter equation, if $(S, r_W)$ is a set-theoretic solution of the Yang-Baxter equation \eqref{eq:yb}.
		
		\begin{theorem}\label{th:weak-sol-var}
			The class of weak distributive solutions of the Yang-Baxter equation is a variety. Moreover, this variety is defined by the identity
			\begin{align*}
			&(x \land y\land x) \lor ((x \lor y \lor x) \land z \land (x \lor y \lor x)) \lor (x \land y\land x)
			\\ &= (x \lor (y\land z \land y) \lor x) \land (y \lor z \lor y) \land (x \lor (y\land z \land y) \lor x).
			\end{align*}
		\end{theorem}
		
		\begin{proof}
			Denote $r=r_W$ and let $x,y,z \in S$. Then
			\begin{align*}
			&(r\times \id) (\id\times r)(r\times \id) (x, y,z) \\&=(r\times \id) (\id\times r)(x\land y \land x,x\lor y \lor x, z) 
			\\ &= (r\times \id)(x \land y\land x, (x \lor y \lor x) \land z \land (x \lor y \lor x), (x \lor y \lor x) \lor z \lor (x \lor y \lor x))
			\\ &= ((x \land y\land x) \land ((x \lor y \lor x) \land z \land (x \lor y \lor x)) \land (x \land y\land x),
			\\ & \qquad (x \land y\land x) \lor ((x \lor y \lor x) \land z \land (x \lor y \lor x)) \lor (x \land y\land x) , 
			\\ & \qquad (x \lor y \lor x) \lor z \lor (x \lor y \lor x))
			\end{align*}
			and
			\begin{align*}
			&(\id\times r)(r\times \id)(\id\times r) (x, y,z) \\&=(\id\times r)(r\times \id)(x, y\land z \land y, y \lor z \lor y)
			\\ &= (\id\times r)(x \land (y\land z \land y) \land x, x \lor (y\land z \land y) \lor x, y \lor z \lor y)
			\\ &= (x \land (y\land z \land y) \land x, 
			\\ & \qquad (x \lor (y\land z \land y) \lor x) \land (y \lor z \lor y) \land (x \lor (y\land z \land y) \lor x), 
			\\ & \qquad (x \lor (y\land z \land y) \lor x) \lor (y \lor z \lor y) \lor (x \lor (y\land z \land y) \lor x)).
			\end{align*}
			
			Using absorption and regularity (\ref{eq:reg1}) we deduce 
			$(x \land y\land x) \land ((x \lor y \lor x) \land z \land (x \lor y \lor x)) \land (x \land y\land x)=x\land y\land z\land y\land x$. Likewise, using  regularity (\ref{eq:reg2}) we deduce $(x \lor y \lor x) \lor z \lor (x \lor y \lor x)=x\lor y\lor z\lor y\lor x$, and using absorption and regularity (\ref{eq:reg2}), $(x \lor (y\land z \land y) \lor x) \lor (y \lor z \lor y) \lor (x \lor (y\land z \land y) \lor x)=x\lor y\lor z\lor y\lor x$.
			
			Thus, the class of weak distributive solutions is defined by the identity
			\begin{align*}
			&(x \land y\land x) \lor ((x \lor y \lor x) \land z \land (x \lor y \lor x)) \lor (x \land y\land x)
			\\ &= (x \lor (y\land z \land y) \lor x) \land (y \lor z \lor y) \land (x \lor (y\land z \land y) \lor x).
			\end{align*}
		\end{proof}

		\begin{lemma}\label{lemma:left-right-weak}
			\begin{itemize}
				\item[(i)] Let $S$ be a left handed skew lattice. Then, given any $x,y\in S$, $r_W(x,y)=r_L(x,y).$
				\item[(ii)] Let $S$ be a right handed skew lattice. Then, given any $x,y\in S$, $r_W(x,y)=r_R(x,y).$
			\end{itemize}
		\end{lemma}
		
		\begin{proof}
			Direct application of left [right]-handedness to the defining identities for weak distributive solutions yields defining identities for left [right] distributive solutions.
		\end{proof}
		
		\begin{theorem}\label{th:weak-sol}
			\begin{itemize}
				\item[(i)] Weak distributive solutions are always idempotent, i.e. $r_W^2=r_W$.
				\item[(ii)] Every strong distributive solution is  a weak distributive solution.
				\item[(iii)] Every distributive, simply cancellative and lower symmetric skew lattice is a weak distributive solution.
				\item[(iv)] [Modulo \emph{Prover9}]
				A skew lattice is a weak distributive solution if and only if it is distributive, simply cancellative and lower symmetric.
			\end{itemize}
		\end{theorem}
		
		\begin{proof}
			Denote by $S_L$ and $S_R$ the left and the right factor of $S$, respectively. 
			
			(i) and (ii) hold because they hold for $S_L$ (where by Lemma \ref{lemma:left-right-weak} $r_W$ reduces to $r_L$) and for $S_R$ (where $r_W$ reduces to $r_R$).
			
			(iii) Let $S$ be a distributive, simply cancellative  and lower symmetric skew lattice. By Theorem \ref{thm:lrhanded} it is enough to prove that both $S_L$ and $S_R$ are weak distributive solutions. Since $S_L$ is left handed, it follows from Lemma \ref{lemma:left-right-weak} that $S_L$ is a weak distributive solution if and only if it is a left distributive solution; likewise, $S_R$ is  a weak distributive solution if and only if it is a right distributive solution. By \cite[Theorem 5.1]{canc}, a left handed skew lattice is left cancellative if and only  if it is lower symmetric and simply cancellative; likewise, a right handed skew lattice is right cancellative if and only if it is lower symmetric and simply cancellative. Hence  $S_L$ is distributive and left cancellative, and thus a left distributive solution by Theorem \ref{th:solution-dist-canc}. Likewise, $S_R$ is distributive and right  cancellative, and thus a right distributive solution by Theorem \ref{th:right-sol-var}.
			
			Prover9 was able to derive a proof of (iv).
		\end{proof}
		
		The skew lattice from Example \ref{ex1} is a weak distributive solution. The associated map $r_W$ is defined by $r_W(x,y)=(x,x)$, for all $x,y \in S$. Thus, $(S,r_W)$ is a right non-degenerate solution.
		
		By \cite[Theorem 5.1]{canc} different kinds of cancellation (left/right/simple/full) coincide in the presence of symmetry. As a consequence, we obtain the following result.
		
		
		\begin{corollary} Let S be a symmetric skew lattice. The following properties are equivalent:
			\begin{enumerate}
				\item $S$ is a left distributive solution.
				\item $S$ is a right distributive solution.
				\item $S$ is a weak distributive solution.
			\end{enumerate}
		\end{corollary}
		
		\begin{proof}
			If a skew lattice is symmetric, then it is left cancellative if and only if it is right cancellative if and only if it is cancellative if and only if it is simply cancellative. By Theorem \ref{th:weak-sol}, a  skew lattice $S$ is a weak distributive solution if and only if it is distributive, lower symmetric and {simply} cancellative; by Theorem \ref{th:solution-dist-canc}, $S$ is a left distributive solution if and only if it is distributive and left cancellative; by Theorem \ref{th:right-sol-var}, $S$ is a right distributive if and only if it is distributive and right cancellative. It follows that all three notions of distributive solutions are equivalent for the class of symmetric skew lattices.
		\end{proof}
		
		One can easily notice that for a lattice, the maps (\ref{eq:defR}), $r_L, r_R, r_W$ coincide. Thus we have the following proposition.
		
		\begin{proposition}
			The following conditions are equivalent for a lattice $(L,\land,\lor)$,
			\begin{enumerate}
				\item $L$ is a strong distributive solution,
				\item $L$ is a left distributive solution,
				\item $L$ is a right distributive solution,
				\item $L$ is a weak distributive solution.
			\end{enumerate}
			One, and thus all of the above conditions are satisfied if and only if the lattice $L$ is distributive. 
		\end{proposition}
		
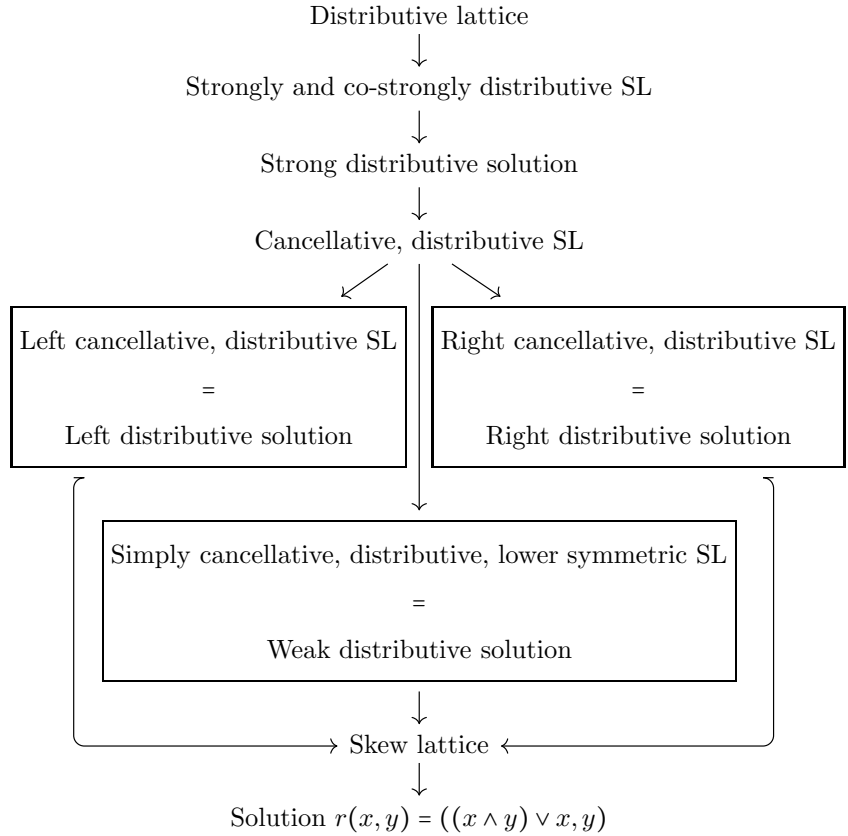
\[\begin{tikzcd}[column sep = -4.35cm, row sep = 0.45cm]
		& \text{Distributive lattice} 
		\arrow{d} & 
		\\ 
		& \text{Strongly and co-strongly distributive SL} 
		\arrow{d} & 
		\\ 
		& \text{Strong distributive solution} \arrow{d} & 
		\\
		& \text{Cancellative, distributive SL} \arrow{dl} \arrow{dd} \arrow{dr} & 
		\\ 
		{\fbox{\Centerstack[c]{Left cancellative, distributive SL\\ $=$ \\ {Left distributive solution}}}} 
		\arrow[rounded corners, to path={ -- ([xshift=-1.795cm]\tikztostart.south) |- (\tikztotarget)}]{rdd} & &  {\fbox{\Centerstack[c]{Right cancellative, distributive SL\\ $=$ \\ {Right distributive solution}}}} {\arrow[rounded corners, to path={ -- ([xshift=1.795cm]\tikztostart.south) |- (\tikztotarget)}]{ldd}}
		\\
		& {\fbox{\Centerstack[c]{Simply cancellative, distributive, lower symmetric SL\\ $=$ \\ {Weak distributive solution}}}} {\arrow{d}} &
		\\
		& {\text{Skew lattice}} {\arrow{d}} &
		\\
		& {\text{Solution $r(x,y) = ((x \land y) \lor x, y)$}} &
		\end{tikzcd}\]
		
		The diagram above gives an overview of all solutions that we discussed in this paper, where we abbreviate skew lattice by SL and the arrows are inclusions between families of skew lattices.
		
		From the diagram, the following corollary is clear.
		\begin{corollary}
			The skew lattice constructed in Proposition \ref{prop:construction} is a left, right and weak distributive solution of the Yang-Baxter equation.
		\end{corollary}
		
		\subsection{Solutions in rings}
		Quadratic  skew lattices in rings are cancellative and distributive by  \cite[Theorems 2.6 and 2.8]{L1}. Cubic skew lattices in rings are cancellative and distributive by \cite[Corollary 5]{Cvet07}. The following pair of results are immediate corollaries of Theorem \ref{th:solution-dist-canc}.
		
		\begin{corollary}
			Let $R$ be a ring and $S\subseteq E(R)$ a multiplicative band that is closed under the operation $\circ$. Then $(S,\cdot,\circ)$ is a left, right and weak distributive solution of the Yang-Baxter equation.
		\end{corollary}
		
		\begin{corollary}
			Let $R$ be a ring and $S\subseteq E(R)$ a multiplicative band such that the operation $\nabla$ is closed and associative on $S$. Then $(S,\cdot,\nabla)$ is a left, right and weak distributive solution of the Yang-Baxter equation.
		\end{corollary}

		\section*{Acknowledgment}
		The first author acknowledges the financial support from the Slovenian Research Agency (research core funding No. P1-0222). The second author is supported by Fonds voor Wetenschappelijk Onderzoek (Flanders), via an FWO Aspirant-mandate.

	\end{document}